\documentclass[12pt]{amsart}
\usepackage{amssymb}
\newtheorem{theorem}{Theorem}[section]
\newtheorem{lemma}[theorem]{Lemma}

\newtheorem{corollary}[theorem]{Corollary}

\theoremstyle{definition}

\newtheorem{examp}[theorem]{Example}

\theoremstyle{remark}

\numberwithin{equation}{section}

\newcommand{\QED}{\qed}
\newenvironment{axiom}{\begin{list}{$\bullet$}{\setlength{\labelsep}{.7cm}%
\setlength{\leftmargin}{2.5cm}\setlength{\rightmargin}{0cm}%
\setlength{\labelwidth}{1.8cm}\setlength{\itemsep}{0pt}}}{\end{list}}
\newcommand{\ax}[1]{\item[{\bf #1}\hfill]\index{#1}}
\newcommand{\diams}{\unskip\nobreak\hfil\penalty50%
\hskip1em\hbox{}\nobreak\hfil%
$\diamondsuit$\parfillskip=0pt\finalhyphendemerits=0}
\newcommand{\bfind}[1]{\index{#1}{\bf #1}}

\newcommand{\n}{\par\noindent}

\newcommand{\sn}{\par\smallskip\noindent}

\newcommand{\bn}{\par\bigskip\noindent}
\newcommand{\pars}{\par\smallskip}

\newcommand{\parb}{\par\bigskip}

\newcommand{\fvklit}[1]{[#1]}

\newcommand{\ovl}[1]{\overline{#1}}

\newcommand{\sep}{^{\rm sep}}

\newcommand{\Gal}{\mbox{\rm Gal}\,}

\newcommand{\cal}{\mathcal}

\newcommand{\N}{\mathbb N}

\begin{document}
\title{Approximation of elements in henselizations}
\author{Franz-Viktor Kuhlmann}
\address{Department of Mathematics and Statistics,
University of Saskatchewan,
106 Wiggins Road,
Saskatoon, Saskatchewan, Canada S7N 5E6}
\email{fvk@math.usask.ca}
\thanks{This work was partially supported by a Canadian NSERC grant and
was completed while I was a guest professor at the Silesian University
at Katowice, Poland. I wish to thank the faculty of mathematics of that
university for their hospitality.}
\date{12.\ 3.\ 2010}
\begin{abstract}\noindent
{\footnotesize\rm
For valued fields $K$ of rank higher than 1, we describe how elements in
the henselization $K^h$ of $K$ can be approximated from within $K$; our
result is a handy generalization of the well-known fact that in rank 1,
all of these elements lie in the completion of $K$. We apply the result
to show that if an element $z$ algebraic over $K$ can be approximated
from within $K$ in the same way as an element in $K^h$, then $K(z)$ is
not linearly disjoint from $K^h$ over $K$.}
\end{abstract}
\maketitle
%
%
%
\section{Introduction}                  
Complete valued fields of rank 1 are henselian, but for valuations $v$ of
arbitrary rank, this does not hold in general. However, there is a
connection between Hensel's Lemma and completions, but these completions
have to be taken for residue fields of suitable coarsenings of $v$. This
connection was worked out by Ribenboim [R] who used \bfind{distinguished
pseudo Cauchy sequences} to characterize the so called \bfind{stepwise
complete} fields; it had been shown by Krull that these fields are
henselian. We want to give a more precise description of this
connection.

Take any extension $(L|K,v)$ of valued fields, that is, an extension
$L|K$ of fields and a valuation $v$ on $L$. By $vL$ and $vK$ we denote
the value groups of $v$ on $L$ and on $K$, and by $Lv$ and $Kv$ the
residue fields of $v$ on $L$ and on $K$, respectively. Similarly, $vz$
and $zv$ denote the value and the residue of an element $z$ under $v$.
For $z\in L$, we define
\[
v(z-L)\>:=\> \{v(z-c)\mid c\in K\}\>\subseteq\> vL\cup\{\infty\}\;.
\]
We call $z$ \bfind{weakly distinguished over $K$} if there is a
non-trivial convex subgroup $\Delta$ of $vK$ and some $\alpha\in vK$
such that the coset $\alpha+\Delta$ is cofinal in $v(z-L)$, that is,
$\alpha+\Delta\subseteq v(z-L)$ and for all $\beta\in v(z-L)$ there is
$\gamma\in \alpha+\Delta$ such that $\beta\leq\gamma$. If this holds
with $\alpha=0$, that is, if some non-trivial convex subgroup of $vK$ is
cofinal in $v(z-L)$, then we call $z$ \bfind{distinguished over
$K$}. This name is chosen since distinguished elements induce
distinguished pseudo Cauchy sequences in the sense of Ribenboim
[R], p.\ 105.

The extension $(L|K,v)$ is \bfind{immediate} if the canonical embeddings
of $vK$ in $vL$ and of $Kv$ in $Lv$ are onto.

Now take an arbitrary valued field $(K,v)$ and extend its valuation $v$
to its algebraic closure $\tilde{K}$. Then $\tilde{K}$ contains a unique
henselization $K^h$ with respect to this extension. We will prove:

\begin{theorem}                             \label{MT1}
Every element $a\in K^h\setminus K$ is weakly distinguished over $K$.
In particular, the henselization is an immediate extension of $(K,v)$.
\end{theorem}

Note that if $(K,v)$ is of rank 1, that is, has archimedean ordered
value group, then its henselization lies in its completion and every
element $a\notin K$ of the henselization $K^h$ is distinguished over
$K$.

We will give two proofs for Theorem~\ref{MT1}. The first one is an
adaptation of the proof found in [Z--S] for the fact that the
henselization of a valued field is an immediate extension. The second
proof uses the fact that the henselization can be constructed as a union
of finite extensions generated by roots of polynomials that satisfy the
conditions of Hensel's Lemma.

By ``$\alpha>v(a-K)$'' we mean $\alpha>v(a-c)$ for all $c\in K$. We use
Theorem~\ref{MT1} to prove the following result:

\begin{theorem}                             \label{MT2}
Take $z\in \tilde{K}\setminus K$ such that
\[v(z-a)\;>\;v(a-K)\]
for some $a\in K^h$. Then $K^h$ and $K(z)$ are not linearly
disjoint over $K$, that is,
\[[K^h(z):K^h] < [K(z):K]\]
and in particular, $K(z)|K$ is not purely inseparable.
\end{theorem}

Theorem~\ref{MT1} answers a question from Bernard Teissier.
Theorem~\ref{MT2} has a crucial application in [Ku3] to the
classification of Artin-Schreier extensions with non-trivial defect.
Theorems~\ref{MT1} and~\ref{MT2} were originally proved in [Ku1], but
the proofs given in Sections~\ref{sectbuild} and~\ref{sectpMT2} are
improved versions of the original proofs, using much less technical
machinery, and the proof of Theorem~\ref{MT1} given in
Section~\ref{sectdeih} is new.

%
%
\section{Some preliminaries}    \label{sectprel}
We will assume the reader to be familiar with the basic facts of
valuation theory, and we will often use them without further references.
We recommend [End], [Eng--P], [R], [W], [Z--S] and [Ku2] for the general
valuation theoretical background.

As we are working with valued fields $(K,v)$ of higher rank (that is,
with non-archimedean ordered value groups $vK$), we will use
convex subgroups $\Delta$ of $vK$ and the corresponding coarsenings of
$v$. The ordering of $vK$ induces an ordering on $vK/\Delta$: the
set of positive elements in the latter group is just the image under the
canonical epimorphism of the set $vK^+$ of positive elements in $vK$.
Hence, $\alpha\geq\beta$ implies $\alpha+\Delta\geq\beta+\Delta$.
More precisely, $\alpha+\Delta\geq\beta+\Delta$ holds if and only if
there is some $\gamma\in\Delta$ such that $\alpha+\gamma\geq\beta$. The
\bfind{coarsening} $v_\Delta$ of $v$ is the valuation whose valuation
ring is $\{c\in K\mid vc\in vK^+\cup\Delta\}$, which contains the
valuation ring ${\cal O}_v$ of $v$. The value group of $v_\Delta$ on $K$
is canonically isomorphic to $vK/\Delta$. We have that $vc\geq vd$
implies $v_\Delta c\geq v_\Delta d$ and in particular, $vc\geq 0$
implies $v_\Delta c\geq 0$.

The valuation $v$ also induces a valuation $\ovl{v}_\Delta$ on the
residue field $Kv_\Delta$ such that $v$ is (equivalent to) the
composition $v_\Delta\circ \ovl{v}_\Delta$ (in this paper, we will
identify equivalent valuations). If ${\cal O}_{v_\Delta}$ and
${\cal M}_{v_\Delta}$ denote the valuation ring and valuation ideal of
$v_\Delta$, then the valuation ring of $\ovl{v}_\Delta$ is the image of
${\cal O}_v$ under the canonical epimorphism ${\cal O}_{v_\Delta}
\rightarrow {\cal O}_{v_\Delta} /{\cal M}_{v_\Delta}=Kv_\Delta$. The
value group of $\ovl{v}_\Delta$ on $Kv_\Delta$ is canonically isomorphic
to $\Delta$ via
\begin{equation}                            \label{iso}
\ovl{v}_\Delta (a+{\cal M}_{v_\Delta}) \>\mapsto\> va
\mbox{ \ \ for } a\notin {\cal M}_{v_\Delta}\;.
\end{equation}

If $(L|K,v)$ is an arbitrary extension of valued fields, then the convex
hull $\Gamma$ of $\Delta$ in $vL$ is a convex subgroup of $vL$, and
$v_\Gamma$ is an extension of $v_\Delta$ from $K$ to $L$. If $vL/vK$ is
a torsion group (which is the case if $L|K$ is algebraic), then taking
convex hulls induces a bijective inclusion preserving mapping from the
chain of convex subgroups of $vK$ to the chain of convex subgroups of
$vL$, and $v_\Gamma$ is the unique coarsening of $v$ on $L$ which
extends $v_\Delta$.

\pars
We will need some facts from ramification theory.

\begin{lemma}                               \label{coardi}
Let $(N|K,v)$ be an arbitrary normal algebraic extension and $w$ a
coarsening of $v$ on $N$. Then
\begin{equation}
(N|K)^{d(w)}\subseteq (N|K)^{d(v)}\subseteq (N|K)^{i(v)}\subseteq
(N|K)^{i(w)}\;,
\end{equation}
where $(N|K)^{d(v)}$ and $(N|K)^{d(w)}$ denote the decomposition fields
of $(N|K,v)$ with respect to $v$ and $w$, respectively, and
$(N|K)^{i(v)}$ and $(N|K)^{i(w)}$ denote the inertia fields
of $(N|K,v)$ with respect to $v$ and $w$, respectively.
\end{lemma}
\begin{proof}
For $\sigma\in\Gal(N|K)$, $v\circ \sigma= v$ implies $w\circ \sigma= w$.
Hence the decomposition group with respect to $v$ is contained in the
decomposition group with respect to $w$. This proves the first
inclusion. The second inclusion is well known from ramification theory
(cf.\ [Eng-P], p.~124). For $\sigma\in\Gal(N|K)$, if $w(x-\sigma x)>0$
for all $x$ such that $wx\geq 0$, then $v(x-\sigma x)>0$ for all $x$
such that $vx\geq 0$. Hence the inertia group with respect to $w$ is
contained in the inertia group with respect to $v$. This proves the
third inclusion.
\end{proof}

\begin{lemma}                               \label{ADE}
Let $(N|K,w)$ be a finite normal extension of valued fields with
decomposition field $Z$ and inertia field $T$. If $z\in T$ then there is
$c\in Z$ such that
\[
w(z-c)\>=\>\max w(z-Z)\>\in wZ\>\;.
\]
\end{lemma}
\begin{proof}
From ramification theory we know that $n:=[T:Z]=[Tw:Zw]$. We choose
$b_1=1,\ldots,b_n\in T$ such that $wb_1=\ldots=wb_n=0$ and $b_1w,
\ldots,b_nw$ is a basis of $Tw|Zw$. Then $b_1,\ldots,b_n$ are
$Z$-linearly independent and thus form a basis of $T|Z$. Since
$b_1w,\ldots,b_nw$ are $Zw$-linearly independent, we have that
$w(c_1b_1+\ldots+c_nb_n)=\min_{1\leq i\leq n} w(c_ib_i)\leq\min_{2\leq
i\leq n}w(c_ib_i)=\min_{2\leq i\leq n}w(c_i)\in wZ$. Hence if
$z=c_1b_1+\ldots+ c_nb_n$ and we set $c=c_1b_1=c_1\in Z$, then $w(z-c)=
\min_{2\leq i\leq n}w(c_ib_i)=\max w(z-Z)$.
\end{proof}

%
%
\section{Properties of weakly distinguished elements} 
Throughout this section, let $(L|K,v)$ be an extension of valued fields.
The extension is immediate if and
only if for every $z\in L\setminus K$ and every $c\in K$ there is $c'\in
K$ such that $v(z-c')>v(z-c)$. This holds if $z$ is weakly distinguished
over $K$ since then, $v(z-K)$ has no maximal element (as a non-trivial
convex subgroup of $vK$ has no maximal element). This proves:

\begin{lemma}                               \label{distimm}
If every $z\in L\setminus K$ is distinguished over $K$, then $(L|K,v)$
is immediate.
\end{lemma}

A subset $S$ of an ordered set $T$ is a \bfind{final segment} of $T$ if
$S\ni\beta<\gamma\in T$ implies $\gamma\in S$, and an \bfind{initial
segment} of $T$ if $S\ni\beta>\gamma\in T$ implies $\gamma\in S$.

\begin{lemma}                               \label{fsegm}
If $v(z-K)$ has no maximal element, then it is an initial segment of
$vK$.
\end{lemma}
\begin{proof}
By our assumption, for every $c\in K$ there is $c'\in K$ such that
$v(z-c)< v(z-c')$, whence $v(z-c)=v(c'-c)\in vK$. This proves that
$v(z-K) \subseteq vK$. If $v(z-c)>\gamma\in vK$, then take $d\in K$ such
that $vd=\gamma$ to obtain that $\gamma=vd=\min\{v(z-c),vd\}=v(z-(c+d))
\in v(z-K)$. This proves that $v(z-K)$ is an initial segment of $vK$.
\end{proof}

If $z\in L$ is distinguished over $K$ with $\Delta$ cofinal in $v(z-K)$,
and if $\Gamma$ is the convex hull of $\Delta$ in $vL$, then for all
$c\in K$, $v(z-c)\geq 0$ implies $v_\Gamma(z-c)\geq 0$ (but the converse
is not true). On the other hand, there is no $c\in K$ such that
$v(x-c)>\Gamma$ because otherwise we would have $v(x-c)>\Delta$. Hence,
$v_\Gamma(z-c)\leq 0$, so $v(z-c)\geq 0$ implies $v_\Gamma(z-c)=0$. We
will denote by $Kv_\Delta^{c(\ovl{v}_\Delta)}$ the completion of
$Kv_\Delta$ with respect to $\ovl{v}_\Delta$.

\begin{lemma}                               \label{chardist}
Suppose that $\Delta$ is a non-trivial convex subgroup of $vK$ and
$\alpha\in vK$ such that $\alpha+\Delta$ is cofinal in $v(z-K)$ (so
that $z$ is weakly distinguished over $K$). Then $z\notin K$, $v(z-K)
\subseteq vK$, and $\alpha+\Delta$ is a final segment of $v(z-K)$.

If in addition $\alpha=0$ (so that $z$ is distinguished over $K$),
$\Gamma$ is the convex hull of $\Delta$ in $vL$, and
$v_\Gamma z=0$, then
\begin{equation}                            \label{comp-rf}
zv_\Gamma\in Kv_\Delta^{c(\ovl{v}_\Delta)} \setminus Kv_\Delta\;.
\end{equation}

Conversely, if there exists a decomposition $v=v_\Gamma\circ
\ovl{v}_\Gamma$ on $L$ such that (\ref{comp-rf}) holds, then $z$ is
distinguished over $K$ with $\Delta=\Gamma\cap vK$ cofinal in $v(z-K)$.
\end{lemma}
\begin{proof}
Suppose that $\Delta$ is a non-trivial convex subgroup of $vK$ and
$\alpha\in vK$ such that $\alpha+\Delta$ is cofinal in $v(z-K)$. Then
$v(z-K)$ has no maximal element, and Lemma~\ref{fsegm} shows that
$v(z-K) \subseteq vK$. In particular, $\infty\notin v(z-K)$, which shows
that $z\notin K$. Since $\Delta$ and hence also $\alpha+\Delta$ is
convex in $vK$, the assumption that $\alpha+\Delta$ is cofinal in
$v(z-K) \subseteq vK$ implies that it is a final segment of $v(z-K)$.

\pars
Now suppose that $\Delta$ is cofinal in (and hence a final segment of)
$v(z-K)$, $\Gamma$ is the convex hull of $\Delta$ in $vL$, and $v_\Gamma
z=0$. Via the isomorphism (\ref{iso}),
let us identify this value group with $\Delta$. Take any $\delta\in
\Delta$. Then we can choose $d\in K$ such that $v(z-d)\in \Delta$ and
$\ovl{v}_\Gamma ((z-d)v_\Gamma)=v(z-d)>\delta$. This yields that
$\ovl{v}_\Gamma (zv_\Gamma - dv_\Delta)>\delta$. Consequently,
$zv_\Gamma\in Kv_\Delta^{c(\ovl{v}_\Delta)}$. On the other hand, if
$zv_\Gamma$ would lie in $Kv_\Delta$ and thus would equal $dv_\Delta$
for some $d\in K$, then we would have that $v_\Gamma(z-d)>0$ and hence
$v(z-d)>\Delta$, a contradiction.

For the converse, let $\Delta$ be any convex subgroup of $vK$ and
$\Gamma$ its convex hull in $vL$, and assume that (\ref{comp-rf}) holds.
Then for every $\delta\in \Delta$  there is $d\in K$
such that $\ovl{v}_\Gamma (zv_\Gamma - dv_\Gamma)=\delta$, that is,
$v(z-d)=\delta$. This shows that $\Delta\subseteq v(z-K)$. But there
is no $d\in K$ such that $v(z-d)>\Delta$ since otherwise,
$zv_\Gamma-dv_\Gamma=(z-d)v_\Gamma=0$ and consequently, $zv_\Gamma\in
Kv_\Delta$.
\end{proof}

\begin{lemma}                               \label{wdcomp}
Take any coarsening $w$ of $v$ on $L$. If $z\in L$ is weakly
distinguished over $K$ with respect to $w$, then also with respect to
$v$.
\end{lemma}
\begin{proof}
Denote by $\Gamma_w$ the convex subgroup of $vL$ associated with the
coarsening $w$, such that $wL=vL/\Gamma_w$. Then $\Delta_w=\Gamma_w\cap
vK$ is the convex subgroup of $vK$ associated with the restriction of
$w$ to $K$, and $wK=vK/\Delta_w$. Further, denote by $\ovl{\Delta}$ the
convex subgroup and by $\ovl{\alpha}$ the element of $wK$ such that
$\ovl{\alpha}+\ovl{\Delta}$ is cofinal in $w(z-K)$. Choose $\alpha\in
vK$ such that $\alpha+\Delta_w=\ovl{\alpha}$. Set $\Delta=\{\delta\in
vK\mid \delta+\Delta_w\in\ovl{\Delta}$; this is a convex subgroup of
$vK$.

We show that $\alpha+\Delta$ is cofinal in $v(z-K)$. Take any
$c\in K$. By assumption there is $\ovl{\delta}\in \ovl{\Delta}$ such
that $\ovl{\alpha}+\ovl{\delta}> w(z-c)=v(z-c)+\Gamma_w$. Take
$\delta\in \Delta$ such that $\delta+\Delta=\ovl{\delta}$; then
$\alpha+\delta>v(z-c)$. On the other hand, for every $\delta\in \Delta$
we can take $\beta\in \Delta$ and some $c\in K$ such that
$\ovl{\alpha}+\ovl{\delta}< \ovl{\alpha}+\ovl{\beta}\leq w(z-c)$. This
implies that $\alpha+\delta<v(z-c)$. This completes our proof.
\end{proof}

We leave the easy proof of the following lemma to the reader.

\begin{lemma}                                      \label{aat}
Take $z\in L$ and $b,c\in K$, $b\ne 0$. Then
\[
v(bz+c-K) \>=\> vb+v(z-K)\;.
\]
Consequently,
\sn
1) \ $bz+c$ is weakly distinguished over $K$ if and only if $z$ is,
\sn
2) \ if $z$ is distinguished over $K$, then $bz+c$ is weakly
distinguished over $K$,
\sn
3) \ if $z$ is weakly distinguished over $K$, then there is some $d\in
K$ such that $dz$ is distinguished over $K$.
\end{lemma}

\begin{lemma}                          \label{dd}
Let $(L|K,v)$ and $(L(z)|L,v)$ be arbitrary extensions of valued fields.
Assume that every element $x\in L\setminus K$ is weakly distinguished
over $K$. If $z$ is weakly distinguished over $L$, then also over $K$.
\end{lemma}
\begin{proof}
From Lemma~\ref{distimm} we know that $vL=vK$. We have that
\[v(z-K) \>\subseteq\>v(z-L)\;.\]
Thus, if ``$\,=$'' holds, we are done. So we assume that ``$\,<$''
holds. Then there exists an element $x\in L$ such that $v(z-c)<v(z-x)$
for every $c\in K$, whence $v(z-c)=v(x-c)$. This shows that
\[v(z-K) \>=\> v(x-K)\;.\]
Since $x$ is weakly distinguished over $K$ by hypothesis, this
shows that also $z$ is weakly distinguished over $K$.
\end{proof}

%
%
\section{Distinguished elements in henselizations}  \label{sectdeih}
Our goal in this section is to show that every element in the
henselization $K^h$ is weakly distinguished over $K$. For valuations of
rank 1, this is a direct consequence of the well known fact that the
completion of a valued field of rank 1 contains its henselization.
Indeed, all elements in this completion and hence also all elements in
the henselization that do not lie in $K$ are distinguished over $K$.

If $(K,v)$ is of rank $>1$, then the extensions of $v$ to a given
algebraic extension field $L$ may not be independent. In this case, the
Strong Approximation Theorem may fail. As a substitute, for the proof
that the henselization is an immediate extension, Ribenboim
[R] gives a generalized version of the Strong Approximation
Theorem where the independence condition is replaced by conditions on
the given data that have to be satisfied by the requested element. But
in our context, the method of Zariski and Samuel \fvklit{ZA--SA2} is
more natural: it proceeds by induction on the number of extensions of
the valuation $v$ and treats dependent extensions by an investigation of
suitable coarsenings of $v$. We adapt this method to prove the
more informative Theorem~\ref{MT1}.

\begin{lemma}                               \label{distindep}
Take a normal separable-algebraic extension $(N|K,v)$ of a valued
fields and assume that all extensions of $v$ from $K$ to $N$ are
independent. Further, assume that $a\in N$ has the property that $v\ne
v\circ \sigma$ on $N$ for every $\sigma\in \Gal(N|K)$ such that $\sigma
a\ne a$. Then $a$ lies in the completion of $(K,v)$.
\end{lemma}
\begin{proof}
Given any $\alpha\in vK$, we have to show that there exists $c\in K$
such that $v(a-c)\geq\alpha$. All extensions of $v$ from $K$ to $N$
are conjugate, that is, of the form $v\circ \sigma$ with $\sigma\in
\Gal(N|K)$ (cf.\ [Eng--P], Theorem~3.2.15). As we assume that all of
them are independent, the same will be true for the finitely many
extensions of $v$ from $K$ to the normal hull $N_a\subseteq N$ of
$K(a)|K$. Moreover, $\sigma a\ne a$ implies that $v\ne v\circ \sigma$
already holds on $N_a$ because otherwise, $v$ and $v\circ \sigma$ are
both extensions of $v=v\circ \sigma$ from $N_a$ to $N$ and hence there
is $\tau\in\Gal(N|N_a)$ such that $v=v\circ \sigma\circ\tau$ on $N$;
but as $\sigma\circ\tau (a)=\sigma a\ne a$, this is a contradiction to
our assumption on $a$.

We use the Strong Approximation Theorem
(cf.\ [Eng--P], Theorem~2.4.1) to find $b\in N_a$ such that
$v(a-b)\geq\alpha$ and $v(\sigma b)=(v\circ \sigma) b\geq\alpha$
whenever $\sigma a\ne a$. Writing $c=\sum_{\sigma}^{} \sigma b$ for the
trace Tr$_{N_a|K}(b)$, we find that
\[
v(a-c)\;\geq\;\min\{v(a-b),v\sigma b\mid\sigma\ne\mbox{\rm id}\}
\;=\;\alpha\;.
\]
\end{proof}

\begin{lemma}                               \label{indec}
The assumption on the element $a$ in Lemma~\ref{distindep} is satisfied
when $a$ lies in the decomposition field $Z$ of $(N|K,v)$.
\end{lemma}
\begin{proof}
If $\sigma a\ne a$ then $\sigma\notin\Gal(N|Z)$. As the latter is the
decomposition group of $(N|K,v)$, this shows that $v\ne v\circ \sigma$
on $N$.
\end{proof}

Take a valued field $(K,v)$ and extend $v$ to the separable-algebraic
closure $K\sep$ of $K$. The henselization $K^h$ of $(K,v)$ is the
decomposition field of $(K\sep|K,v)$ (cf.\ [Eng--P], Theorem~5.2.2).
From the two preceding lemmata, we obtain:

\begin{corollary}                           \label{distindepcor}
If $(N|K,v)$ normal separable-algebraic extension of valued fields and
all extensions of $v$ from $K$ to $N$ are independent, then the
decomposition field of $(N|K,v)$ is contained in the completion of
$(K,v)$. If all extensions of $v$ from $K$ to $K\sep$ are independent
(which in particular is the case if the rank of $(K,v)$ is 1), then
$K^h$ is contained in the completion of $(K,v)$.
\end{corollary}

Now we are ready for the
\sn
{\bf Proof of Theorem~\ref{MT1}}:\n
Since $K^h$ is the decomposition field of $(K\sep|K,v)$, it follows that
for every normal separable-algebraic extension $(N|K,v)$, the
decomposition field is $K^h\cap N$ (cf.\ [End], (15.6) c)\,). Hence
$K^h$ is the union over the decomposition fields of all finite normal
separable-algebraic extensions of $(K,v)$. Thus we may assume that $a$
lies in the decomposition field $(Z,v)$ of some finite Galois
extension $(N|K,v)$. Let $v_1=v,v_2,\ldots,v_n$ be
all extensions of $v$ from $K$ to $N$. (Note that $n\geq 2$ because
the assumption $a\notin K$ implies that $Z\ne K$.)

If $n\geq 3$, then suppose that the lemma is already proved for the case
where the number
of extensions of the valuation $v$ from $K$ to $N$ is smaller than $n$.
In view of Corollary~\ref{distindepcor}, we only have to treat the case
where the extensions $v_1,\ldots,v_n$ are not independent on $N$. Hence,
there are $i,j$ such that $v_i$ and $v_j$ admit a nontrivial common
coarsening. The restriction of this coarsening to $K$ is also a
nontrivial coarsening of the valuation $v$ on $K$. (Indeed, as $N|K$ is
algebraic, restriction induces an inclusion preserving bijection between
the coarsenings of $v_i$ and the coarsenings of $v$ on $K$ which
preserves inclusion between the corresponding valuation rings.)
Among all the coarsenings of $v$ on $K$ that we find in this way,
running through all common coarsenings of all possible pairs $v_i$ and
$v_j$, let $w$ be the finest one. (Its valuation ring is the
intersection of the valuation rings of all of these coarsenings.) We
write $v=w\circ \ovl{w}$. Now $w$ admits an extension (also called $w$)
to $N$ which is a coarsening of at least two of the $v_i$'s.

W.l.o.g., we may assume that $w$ is also a coarsening of $v_1=v$.
Indeed, since all extensions of $v$ from $K$ to $N$ are conjugate, we
may choose $\sigma\in \Gal(N|K)$ such that $v_i\circ\sigma=v_1$, and we
obtain that $w\circ\sigma$ is an extension of $w$ from $K$ to $N$ and a
coarsening of $v_i\circ\sigma=v_1$ and of $v_j\circ\sigma\not=v_1$.

For the coarsening $w$ of $v$, we may infer from Lemma~\ref{coardi},
using the notation of that lemma:
\[(N|K)^{d(w)}\subset (N|K)^{d(v)}\subset (N|K)^{i(v)}\subset
(N|K)^{i(w)}\;.\]
We set $L=(N|K)^{d(w)}$; note that $Z=(N|K)^{d(v)}$.

Every extension of $w$ from $K$ to $N$ may be refined to an extension of
$v$ from $K$ to $N$ (just by composing it with any extension of
$\ovl{w}$ from $Kw$ to $Nw$). Since the extension $w$ gives already rise
to at least two extensions of $v$ from $K$ to $N$, we see that there
cannot be more than $n-1$ extensions of $w$ from $K$ to $N$. By our
induction hypothesis, we find that every element $a\in L\setminus K$ is
weakly distinguished over $K$ with respect to $w$, and by
Lemma~\ref{wdcomp}, also with respect to $v$. In view of Lemma~\ref{dd},
it now suffices to show that every element $z\in Z\setminus L$ is weakly
distinguished over $L$.

Since $Z$ is contained in $(N|K)^{i(w)}$, we may infer from
Lemma~\ref{ADE} the existence of an element $c\in L$ such that
$w(z-c)=\max w(z-L)\in wL$. We choose $b\in L$ such that $wb(z-c)=0$. By
Lemma~\ref{aat}, $z$ is weakly distinguished over $L$ if and only if
$b(z-c)$ is. Consequently, we may assume $c=0$, $b=1$ and
\[
0\>=\>wz\>=\> \max w(z-L)
\]
from the start.

After a suitable renumbering, we may assume that precisely the
extensions $v_1=v,v_2,\ldots,v_m$ of $v$ are composite with $w$, and we
may write $v_j=w\circ\ovl{w}_j$ for $1\leq j\leq m$. Now $(Z,v)$ is also
the decomposition field of $(N|L,v)$ (cf.\ [End], (15.6) b)\,). Since
$L$ was chosen to be the decomposition field of $(N|K,w)$, the extension
of $w$ from $L$ to $N$ is unique. Every $\tau\in \Gal (Nw|Lw)$ is
induced by some $\sigma\in \Gal(N|L)$. If $\tau (zw)\ne zw$, then
$\sigma z\ne z$, and by Lemma~\ref{indec}, $v\circ \sigma \not= v$ while
$w\circ \sigma = w$ on $N$. This implies that $\ovl{w}_1\circ \tau\not=
\ovl{w}_1$ on $Nw$.

Furthermore, by our choice of $w$, it is the finest coarsening of $v$ on
$K$ which is induced by a common coarsening of at least two $v_i$'s.
Consequently, the $\ovl{w}_i$'s must be independent since otherwise, a
common nontrivial coarsening of them could be composed with $w$ to
obtain a finer valuation, in contradiction to our choice of $w$. We have
thus shown that the extension $(Zw|Lw, \ovl{w_1})$ satisfies the
hypotheses of Lemma~\ref{distindep}. We conclude that $zw$ lies in the
completion of $(Lw,\ovl{w}_1)$. On the other hand, $\max w(z-L)=0$ shows
that there is no element $c\in L$ such that $w(z-c)>0$. This proves
$zw\notin Lw$. Hence by Lemma~\ref{chardist}, $z$ is distinguished over
$L$.

Now the second assertion of Theorem~\ref{MT1} follows from
Lemma~\ref{distimm}.                                         \QED

%
%
\section{Building up the henselization}     \label{sectbuild}
We will give a different approach to the proof of Theorem~\ref{MT1}. It
starts with the following observation:
\begin{lemma}                               \label{hedd}
Let $(K,v)$ be an arbitrary valued field and $f\in {\cal O}_v[X]$ be
non-linear, monic and irreducible over $K$. Assume that $a\in\tilde{K}$
is a root of $f$ such $av\in Kv$ and $vf'(a)=0$. Then $a$ is
distinguished over $K$.
\end{lemma}
\begin{proof}
From the Taylor expansion we infer the existence of some
$\tilde{h}(X,Z)\in {\cal O}_v[X,Z]$ such that
\[
f(Z)-f(X)=f'(X)(Z-X)+(Z-X)^2\tilde{h}(X,Z)\;.
\]
Since $av\in Kv$, there is $c\in {\cal O}_v$ such that $v(a-c)>0$. Given
any such $c$, we note that $vf'(c)=vf'(a)=0$, and we set
\begin{equation}                            \label{refine}
c'\>:=\>c-\frac{f(c)}{f'(c)}\>\in\> {\cal O}_v\;.
\end{equation}
Then
\[
f(c')-f(c)\>=\>f'(c)(c'-c)+(c'-c)^2\tilde{h}(c',c)
\>=\>-f(c)+\frac{f(c)}{f'(c)}^2\tilde{h}(c',c)
\]
with $\tilde{h}(c',c)\in {\cal O}_v$, so that
\[
vf(c')\>=\>2vf(c)+v\tilde{h}(c',c)\>\geq\>2vf(c)\;.
\]
On the other hand,
\[
f(c)\>=\>f(c)-f(a)\>=\>f'(a)(c-a)+(c-x)^2\tilde{h}(c,a)\;.
\]
Since $vf'(a)=0$, $v(c-a)^2>v(c-a)$ and $v\tilde{h}(c,a)\geq 0$, it
follows that
\[
vf(c)\>=\>v(c-a)\>=\>v(a-c)\> >\>\;.
\]
Now (\ref{refine}) implies that $v(c'-c)>0$, whence $v(c'-a)>0$.
Replacing $c$ by $c'$ in the above argument, we may thus deduce that
\begin{equation}                            \label{x2}
v(a-c')\>=\>v(c'-a)\>=\>vf(c')\>\geq\>2cf(c)\>=\>2v(a-c)\;.
\end{equation}
First of all, this yields that $v(a-K)$ has no maximal element (note
that $\infty\notin v(a-K)$ as $a\notin K$ by our assumption on $f$).
Hence by Lemma~\ref{fsegm}, we know that the non-empty set of positive
elements in $v(a-K)$ is convex in $vK$. Therefore, (\ref{x2}) yields
that it is closed under addition and thus the set of positive elements
of a convex subgroup. This proves that $a$ is distinguished over $K$.
\end{proof}

We will consider a very special type of immediate extensions
$(K(z),v)|(K,v)$, and build up the henselization by a transfinite
repetition of such extensions. We call an element $z$ \bfind{strictly
distinguished} over $K$ if there exists a coarsening $w$ of $v$ such
that the following three conditions hold:

\begin{axiom}
\ax{(SD1)} $wz = 0\,$,
\ax{(SD2)} $zw\in Kw^{c(\ovl{w})}\setminus Kw\,$,
\ax{(SD3)} for all $n\in\N\,$, if $1,z,\ldots,z^n$ are linearly
independent over $K$, then\\ $1,zw,\ldots,(zw)^n$ are linearly
independent over $Kw\,$.
\end{axiom}
The third condition implies that $[K(z):K]=
[Kw(zw):Kw]$; in particular, if $z$ is transcendental over $K$, then
$zw$ is transcendental over $Kw$. Lemma~\ref{chardist} shows that if
$z$ is strictly distinguished over $K$, then $z$ is distinguished
over $K$.

The next lemma shows that strictly distinguished elements generate
extensions with a nice property. For $f\in {\cal O}_K[X]$, we denote by
$fv$ the polynomial obtained from $f$ by replacing the coefficients by
their $v$--residues.

\begin{lemma}                       \label{sd}
Let $z$ be strictly distinguished over $K$. Then every element $y\in
K(z)\setminus K$ is weakly distinguished over $K$.
\end{lemma}
\begin{proof}
Let the decomposition $v=w\circ\ovl{w}$ be as in the above definition of
strictly distinguished elements. In the first step, we will prove the
lemma under the assumption that $y=f(z)$ with $f\in K[X]$ and $\deg f
<[K(z):K]$ if the latter is finite. (If $z$ is algebraic over $K$, then
this assumption is no loss of generality.) By Lemma~\ref{aat}, for every
$b\in K^{\times}$ and $c\in K$ we have that $y$ is weakly distinguished
over $K$ if and only if $by-c$ is; after picking suitable elements $b,c$
and replacing $f$ by $bf-c$ we may thus assume that $f$ has no constant
term and that $f\in {\cal O}_{(K,w)} [X] \setminus {\cal M}_{(K,w)}[X]$.
Consequently, $fw \not\equiv 0$, and since $wz=0$, we have $f(z)w =
(fw)(zw)$. By our assumption on the degree of $f$, the elements $1,z,
\ldots ,z^{\mbox{\scriptsize deg} f}$ are linearly independent over $K$,
and by condition (SD3), the same holds for the elements $1,zw,\ldots,
(zw)^{\mbox{\scriptsize deg} f}$ over $Kw$. Hence $(fw)(zw)\notin Kw$.
But since $zw$ is an element of the completion of $(Kw,\ovl{w})$, the
element $f(z)w = (fw)(zw)$ also lies in the completion of
$(Kw,\ovl{w})$. In view of Lemma~\ref{chardist}, this shows $f(z)$ to be
weakly distinguished over $K$.

In the second step, it remains to prove the lemma for the case where
$z$ is transcendental over $K$ and $y=f(z)/g(z)$ with $f,g\in K[X]$. By
a similar argument as above, after multiplication of $f$ and $g$ (and
hence of $y$) with suitable elements from $K^{\times}$, we may assume
that $f,g\in {\cal O}_{(K,w)}[X]\setminus {\cal M}_{(K,w)}[X]$. To avoid
the case where $(f(z)/g(z))w = (f(z)w)/(g(z)w)\in Kw$, we have to do the
following. If $m = \deg gw$, then the $m$-th coefficient of $g$ is not
zero; hence there exists an element $d\in K$ such that the $m$-th
coefficient of the polynomial $f-dg$ is 0. Again, after multiplication
of $f-dg$ with a suitable element from $K^{\times}$, we may assume that
$f-dg\in{\cal O}_{(K,w)}[X]\setminus {\cal M}_{(K,w)}[X]$. Then
\[
\frac{(f(z)-dg(z))w}{g(z)w}\notin Kw\;,
\]
but this element lies in the completion of $(Kw,\ovl{w})$ since the same
holds for $(f(z)-dg(z))w$ and $g(z)w$. Since $(f-dg)/g = (f/g) - d$,
it follows by Lemma~\ref{chardist} and Lemma~\ref{aat} that $y=f(z)/
g(z)$ is weakly distinguished over $K$.
\end{proof}

The next lemma shows how strictly distinguished elements appear in
henselizations.

\begin{lemma}                               \label{hensstr}
Let $(K,v)$ be an arbitrary valued field and $f\in {\cal O}_v[X]$ be
non-linear, monic and irreducible over $K$. Assume that $a\in\tilde{K}$
is a root of $f$ such that $av$ is an element of $Kv$ and a simple root
of $fv$. Then $a$ is distinguished over $K$. If in addition, for every
proper coarsening $w$ of $v$ either $fw$ remains irreducible over $Kw$
or admits a root in $Kw$ with $\ovl{w}$--residue $av$, then $a$ is
strictly distinguished over $K$.
\end{lemma}
\begin{proof}
The first part of the lemma follows directly from Lemma~\ref{hedd} via
a reformulation of the condition on $a$. Now let $f$ satisfy the
hypothesis of the second part. By the first part of the lemma,
$a$ is distinguished over $K$. By virtue of $f\in {\cal O}_v[X]$
we have $va\geq 0$. An application of Lemma~\ref{chardist} thus
yields $v_\Gamma a= 0$ and $av_\Gamma\in Kv_\Delta^{c(\ovl{v}_\Delta)}
\setminus Kv_\Delta$, with $\Gamma$ and $\Delta$ as in that lemma. It
remains to show that $a$ also satisfies condition (SD3) for $w=
v_\Gamma$. Since $av$ is a simple root of $fv$, we know that $av_\Gamma$
is the only root of $fv_\Gamma$ with $\ovl{w}$--residue $av$. On the
other hand, $av_\Gamma\notin Kv_\Delta$, and our hypothesis now yields
that $fv_\Gamma$ is irreducible over $Kv_\Delta$. Since $f$ is monic, we
now have $[Kv_\Delta (av_\Gamma):Kv_\Delta]=\mbox{\rm deg} fv_\Gamma
=\mbox{\rm deg} f = [K(a):K]$ which yields (SD3) for $w=v_\Gamma$.
\end{proof}

The additional condition on the polynomial $f$ that we have introduced
in the above lemma is not too restrictive:
\begin{lemma}
The valued field $(K,v)$ is henselian if and only if it satisfies the
following ``weaker'' version of Hensel's Lemma:\n
Let $f\in {\cal O}_v[X]$ be monic and $a\in K$ such that $fv$
admits $av$ as a simple zero. Assume in addition that $fw$ admits a
root with $\ovl{w}$--residue $av$ for every proper coarsening $w$ of
$v$ for which $fw$ is reducible. Then $f$ admits a root in $K$ with
residue $av$.
\end{lemma}
\begin{proof}
We have to show the above version implies the original version of
Hensel's Lemma (the one without the additional assumption). Assume that
$(K,v)$ is not henselian. Then there is some polynomial $g\in {\cal
O}_v[X]$ having no root in $K$, and $a\in K$ such that $gv$ admits
$av$ as a simple zero. Consider all coarsenings $w$ of $v$, such
that $gw$ admits a factor $g_w$, irreducible over $Kw$ and of degree
$>1$, and such that the $\ovl{w}$--reduction $g_w\ovl{w}$ admits $av$ as
a zero. Among these, we choose a coarsening $w_0$ for which $g_{w_0}$
has least degree. Furthermore, we choose any $f\in {\cal O}_v[X]$ with
$fw_0 = g_{w_0}$ and $\mbox{\rm deg} f = \mbox{\rm deg} g_{w_0}$. Then
$f$ satisfies the above condition: $fv$ admits $av$ as a simple zero,
and for every coarsening $w$ of $v$, the polynomial $fw$ is either
irreducible or admits a zero whose $\ovl{w}$--residue is equal to $av$.
But $f$ does not admit any root in $K$ since its $\ovl{w}$--reduction
$g_{w_0}$ is irreducible over $Kw$ and of degree $>1$. This shows that
$(K,v)$ does not satisfy the above version of Hensel's Lemma.
\end{proof}

The henselization $K^h$ can be generated over $K$ by a transfinitely
repeated adjunction of roots $x$ of polynomials which satisfy the
hypothesis of Hensel's Lemma. The foregoing lemma shows that this is
also true if we replace Hensel's Lemma by the above version. In this
case, in every step an element is adjoined which is strictly
distinguished over the previous field, according to Lemma~\ref{hensstr}.
The next lemma shows why we are choosing this procedure.

\begin{lemma}
Let $(M|K,v)$ be an extension of valued fields generated by a set of
elements $\{z_{\nu}\mid\nu < \tau\}\subset M$, where $\tau$ is an
ordinal number, such that for every $\nu<\tau$, the element $z_{\nu}$ is
strictly distinguished over $K_{\nu}:=K(z_{\mu}|\mu<\nu)$ (where
$K_0:=K$). Then every element $z\in M\setminus K$ is weakly
distinguished over $K$.
\end{lemma}
\begin{proof}
We prove the lemma by transfinite induction on $\rho<\tau$. The
assertion holds trivially for the field $K$. Now assume $\rho\geq 1$ and
that the assertion holds for every $K_{\mu}$ with $\mu<\rho$. If $\rho$
is a limit ordinal, then $K_{\rho} = \bigcup_{\mu<\rho} K_{\mu}$ showing
that the assertion holds for $K_{\rho}$ too. Now let $\rho=\nu+1$ be a
successor ordinal. Then $K_{\rho}=K_{\nu}(z_{\nu})$ where $z_{\nu}$ is
strictly distinguished over $K_{\nu}$. Let $y$ be an arbitrary element
of $K_{\nu} (z_{\nu})\setminus K_{\nu}$. By Lemma~\ref{sd}, $y$ is
weakly distinguished over $K_{\nu}$. By our induction hypothesis, every
element $x\in K_{\nu}\setminus K$ is weakly distinguished over $K$. In
view of Lemma~\ref{dd}, this yields that also $y$ is weakly
distinguished over $K$. Hence, the lemma holds for $K_{\rho}$, and the
induction step is established.
\end{proof}

This lemma and Lemma~\ref{hensstr} yield the following corollary, which
together with Lemma~\ref{distimm} again proves Theorem~\ref{MT1}:
\begin{corollary}
Let $K$ be a valued field. The henselization $K^h$ can be generated over
$K$ in the way as described in the hypothesis of the foregoing lemma.
Thus, every element in $K^h\setminus K$ is weakly distinguished over
$K$.
\end{corollary}

%
%
\section{Proof of Theorem~\ref{MT2}}        \label{sectpMT2}
We need the following lemma. We assume that the valuation $v$ of a field
$K$ is extended to its algebraic closure $\tilde{K}$ and there has a
decomposition $v=w\circ \ovl{w}$.

\begin{lemma}                               \label{l}
If $(K(y)|K,v)$ is an algebraic extension and such that $wy=0$ and
$yw\in K^h w\setminus Kw$, then $K^h$ and $K(y)$ are not linearly
disjoint over $K$.
\end{lemma}
\begin{proof}
Since $K^h$ is henselian for the valuation $v$, it is
also henselian for the coarsening $w$ (because if $w$ would admit two
distinct extensions to the algebraic closure of $K$ then we could use
them to construct two distinct extensions of $v$).

Let $f(X) \in K[X]$ be the minimal polynomial of $y$ over $K$. Our
assertion is proved if we are able to show that $f$ is reducible over
$K^h$. At this point, we may assume that all conjugates of $y$ over $K$
have the same value $vy$ since otherwise, the inequality $[K^h(y):K^h] <
[K(y):K]$ is immediately seen to be true. This assumption yields $f\in
{\cal O}_w[X]$, and because $f$ is monic, its reduction $fw$ is
non-trivial. The minimal polynomial $g\in Kw[X]$ for $yw$ over $Kw$ has
degree $>1$ since $yw \notin Kw$. Furthermore, it must divide $fw$ which
satisfies $(fw)(yw) = f(y)w =0$. From Lemma~\ref{coardi}, we infer that
$K^h = (K\sep|K)^{d(v)}$ lies in $L:=(K\sep|K)^{i(w)}$. Since $Lw|Kw$ is
separable (cf.\ [Eng-P], Theorem~5.2.7.(1)\,), we find that $xw$ is a
simple root of $g$.

Applying Hensel's Lemma to the henselian field $(K^h,w)$, we conclude
that $f$ becomes reducible over $K^h$; indeed, $f$ factors into two
nontrivial polynomials where the roots of the first one all have
$w$--residue $yw$ while there exists at least one root (in $\tilde{K}$)
of the second polynomial which has as $w$--residue a root of $g$ (in
$\tilde{K}w$) which is different from $yw$. This proves our lemma.
\end{proof}

An alternative proof of this lemma reads as follows. Again, we use
$K^h w = Kw^{h(\ovl{w})}$. By the hypothesis of the lemma,
$(Kw(yw)|Kw,\ovl{w})$ is thus a nontrivial subextension of
$(Kw,\ovl{w})^h | (Kw,\ovl{w})$. By general ramification theory, it
admits at least two extensions of the valuation $\ovl{w}$ from $Kw$ to
$Kw(yw)$. Since these give rise to different extensions of the valuation
$v$ from $K$ to $K(y)$, it again follows from general ramification
theory that $K^h$ and $K(y)$ are not linearly disjoint over $K$.

\parb
With the help of this lemma and Theorem~\ref{MT1}, we are now able to
give the
\sn
{\bf Proof of Theorem~\ref{MT2}}:\n
Assume that $(K,v)$ is any valued field, $v$ is extended to $\tilde{K}$,
$z\in \tilde{K}\setminus K$ and $a\in K^h$ such that $v(z-a)>v(a-K)$.
Then $a\notin K$ since otherwise, $\infty\in v(a-K)$ and $v(z-a)>
\infty$, a contradiction.

Since $a\in K^h\setminus K$, Theorem~\ref{MT1} shows that $a$ is weakly
distinguished over $K$. By Lemma~\ref{aat}, there is $d\in K^\times$ and
a convex subgroup $\Gamma$ of $v\tilde{K}$ which is cofinal in $v(da-K)
=vd+v(a-K)$. By Lemma~\ref{chardist}, $(da)v_\Gamma\notin Kv_\Gamma$. As
$v(z-a)>v(a-K)$ implies that $v(dz-da)>vd+v(a-K)$, we find that
$v(dz-da)>\Gamma$. Thus, $(dz)v_\Gamma=(da)v_\Gamma \in K^h
v_\Gamma\setminus Kv_\Gamma$. Now Lemma~\ref{l} shows that $K^h$ and
$K(z)=K(dz)$ are not linearly disjoint over $K$.
\QED

\bn
\newcommand{\lit}[1]{\bibitem #1{#1}}

\end{document}